\documentclass{amsart}

\usepackage{amsthm}
\usepackage{amscd}
\usepackage{amssymb}
\usepackage{mathrsfs}
\usepackage[xdvi]{graphics}
\usepackage{graphicx}
\usepackage{wrapfig}

\theoremstyle{plain}
\newtheorem{lem}{Lemma}
\newtheorem{thm}{Theorem}
\newtheorem{prop}{Proposition}

\theoremstyle{definition}
\newtheorem{defn}{Definition}

\newcommand{\N}{\mathbb{N}}
\newcommand{\Z}{\mathbb{Z}}

\newcommand{\C}{\mathbb{C}}

\newcommand{\K}{\mathscr{K}}

\newcommand{\T}{\mathscr{T}}
\renewcommand{\O}{\mathcal{O}}

\begin{document}

\begin{abstract} In \cite{BNW} we introduced the notion of a partial translation $C^*$-algebra for a discrete metric space. Here we demonstrate that several important classical $C^*$-algebras and extensions arise naturally by considering partial translation algebras associated with subspaces of trees.
\end{abstract}

\title{Partial translation algebras for trees}

\author{J. Brodzki}
\author{G.A. Niblo}
\author{N.J. Wright}

\maketitle

\section{Introduction}

The uniform Roe algebra $C^*_u(X)$ is a $C^*$-algebra associated with any discrete metric space $X$ which encodes analytically the coarse geometry of the space  \cite[Ch.~4]{R}. For example, the space $X$ has Yu's property $A$ \cite{Yu} if and only if the uniform Roe algebra of $X$ is nuclear \cite{STY}.  
A rich source of examples of interesting metric spaces is the class of discrete groups equipped with a left invariant metric. For such a group $G$, the uniform Roe algebra contains the (right) reduced $C^*$-algebra of the group $C^*_\rho(G)$. The uniform Roe algebra is vastly larger than $C^*_\rho(G)$, indeed, unless $G$ if finite, it is not separable. For this reason, it is useful for general metric spaces to consider an analogue of the reduced group $C^*$-algebra. 
In \cite{BNW} we introduced the notion of a partial translation algebra for a discrete metric space to play this role.  

In this paper we demonstrate the power of this approach by exhibiting several well known algebras in the new  framework.
In the context of subspaces of $\mathbb Z$ the partial translation algebra encodes the additive structure. We provide several examples of this phenomenon. In particular the Toeplitz extension arises here by considering the algebra associated to the inclusion of the natural numbers in the integers (Theorem \ref{Toeplitz}). Replacing the natural numbers by $\mathbb Z\setminus \{0\}$ we obtain a trivial extension of $C(S^1)$ by the compacts which is therefore inequivalent to the Toeplitz extension (Theorem \ref{Trivial}). It is also natural to ask what the associated algebra tells us about the additive structure of the primes, and here we make a connection with the classical de Polignac conjecture (Theorem \ref{Primes}). 

By generalising these ideas to consider the inclusion of the 3-valent tree and the rooted 3-valent tree into the Cayley graph of the free group on two generators we are able to recover the extension used by Cuntz in his computation in \cite{Cuntz1} of the $K$-theory of the Cuntz algebra $\mathcal {O}_2$ (Theorem \ref{Cuntz}). It is straightforward to generalise this method for the algebras $\mathcal{O}_n$, by considering the Cayley graph of the free group $F_n$ on $n$ generators, and we give an outline of this.

Finally we use the geometry to construct an explicit embedding of $C^*_\rho(F_2)$ into the Cuntz algebra $\O_2$ (Theorem \ref{CuntzEmbedding}). We do this by exhibiting explicit injective quasi-isometries of the regular 4-valent tree into the regular $3$-valent tree which are well behaved with respect to the natural partial translations on these trees.

\section{Partial translations}

Let $G$ be a discrete group equipped with a left invariant metric $d$. This means that for every element $l\in G$ the map on $G$ defined using the left multiplication by $l$ is an \emph{isometry} with respect to the metric $d$. 
On the other hand, the right multiplication by an element $r \in G$ moves every element $g\in G$ by the same amount: 
\[
d(g,gr) = d(e,r). 
\]
By analogy with metric geometry, we will call such maps \emph{translations}. These two actions together are responsible for the symmetries and the homogeneity of the group $G$ regarded as a metric space with respect to the left invariant  metric $d$. 

It is clear that one cannot hope to have the same amount of information for a general discrete metric space $X$. However, in \cite{BNW}, we introduced a way of measuring homogeneity. A starting point of our investigations was the observation that, up to a bounded amount of distortion, 
a metric subspace of a discrete group retains a degree of symmetry. Moreover, this induced structure can be codified.

\begin{wrapfigure}{r}{5cm}
\includegraphics[width=5cm]{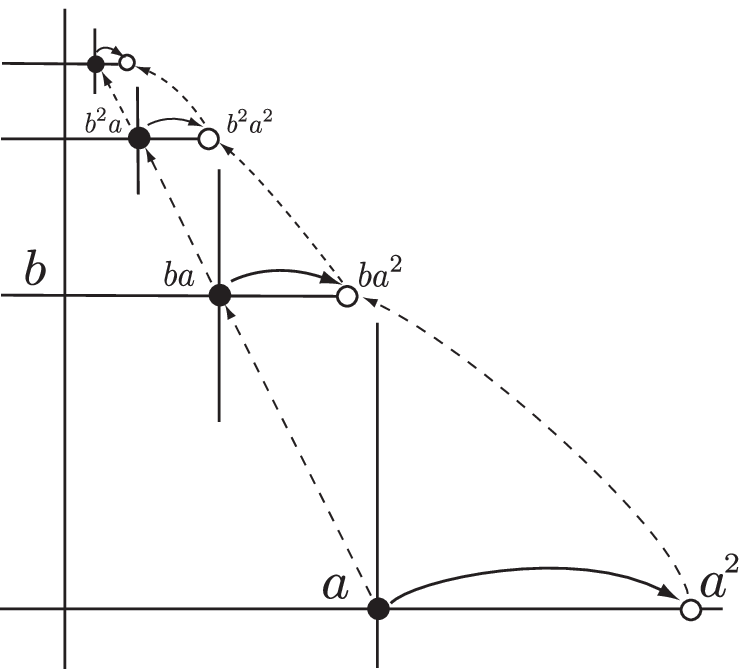}
\begin{center}
\begin{minipage}{5cm}{
\footnotesize  Solid arrows describe the partial translation $\cdot a$ acting on the subset $\{b^na \mid n\in \mathbb{Z}\}$ of the free group 
$\mathbb{F}_2$; dashed arrows denote left multiplication by $b\cdot$.}
\end{minipage}\end{center}
\vfill
\end{wrapfigure}

By a \emph{partial translation}  we mean a bijection $t$ defined on a subset $S\subseteq X$, taking values in a subset of $X$, such that 
$d(s, t(s))$ is bounded for all $s\in S$. This notion was introduced by Roe  \cite[Def.~10.21]{R} in his discussion of the 
coarse groupoid of Skandalis, Tu and Yu \cite{STY}. Equivalently, a partial translation may be described in terms of its graph; from that point of view it may be defined 
as a subspace of $X\times X$ which lies within a bounded distance of the diagonal and such that the coordinate projections are injective. 
In the case when $X$ is a discrete group, right multiplication by an element $g\in X$ determines 
a partial translation $t_g: y\mapsto yg$, which is defined for every $y\in X$.

\begin{defn} Let $X$ be a uniformly discrete bounded geometry space, and let $\bf T$ be a family of disjoint partial translations on $X$. Each of the partial translations $t\in \bf T$ induces a partial isometry $\tau$ on $\ell^2(X)$ defined by $\tau(\delta_x)=\delta_{t(x)}$ for $x$ in the domain of $t$, and $\tau(\delta_x)=0$ for $x$ not in the domain.

The \emph{partial translation 
algebra} $C^*(\bf T)$ is the $C^*$ sub-algebra of the uniform Roe algebra, $C^*_u(X)$, generated by the partial translations  in $\bf T$, regarded as partial isometries in $\ell^2(X)$ (see \cite[p. 67]{Roe} for a discussion of $C^*_u(X)$). 

Note that for any partial translation $t$ the inverse $t^{-1}$ is also defined as a partial translation on $X$ and that as an element of $\ell^2(X)$, it induces the adjoint $\tau^*$ of $\tau$.
\end{defn}

The notion of a partial translation algebra was introduced to play the role of the reduced group $C^*$-algebra for a discrete metric space, and in \cite[Thm 27]{BNW} we prove that in a countable discrete group there is a canonical family of partial translations  $\bf T$ such that  the algebra $C^*(\bf T)$ is isomorphic to $C^*_\rho(G)$.  

In general, without additional constraints on the partial translations, we do not expect to be able to recover (or use) geometric information. However, as we showed in \cite{BNW}, in many cases, and in particular in the case of a subspace of a discrete group, we can choose our partial translations to satisfy strong (partial) homogeneity conditions which control the structure of the corresponding partial translation algebras and relate this to the geometry of the space. The analytic properties of this algebra capture some interesting metric properties of the space $X$. One example of this relation is the statement that
when $X$ is sufficiently homogeneous  (i.e., in the language of \cite{BNW} it admits a free and globally controlled atlas for some partial translation structure), then the following statements are all equivalent \cite[Thm 29]{BNW}: 
\begin{enumerate}
\item The space $X$ has property $A$; 
\item The uniform Roe algebra $C^*_u(X)$ is nuclear; 
\item The algebra $C^*_u(X)$ is exact. 
\end{enumerate}
The conditions of this statement are satisfied, for example, when $X$ admits an injective uniform embedding into  a countable discrete group.

In this paper we will consider subspaces of trees, which of course embed in free groups and therefore inherit well behaved partial translations.

\section{Translation subspaces of $\Z$}




Coburn's theorem, \cite{C},  states that the Toeplitz algebra is the middle term of an extension:

$$0 \to \K\to \T \to C(S^1) \to 0,$$


\noindent where $\K$ denotes the compacts. This extension arises naturally by viewing the generator of the Toeplitz algebra as the unilateral shift on $\ell^2(\N)$ and  identifying the generator of $C(S^1)$ with the biliteral shift on $\ell^2(\Z)$ , induced by the identification of  $C(S^1)$ with the reduced $C^*$ algebra of $Z$. The point of introducing it here is that, as we shall see, it yields the first non-trivial example of a partial translation algebra, arising from the inclusion of $\N$ in $\Z$. 

\subsection{The translation algebra $C^*(\N)$}

In this section we establish the following:

\begin{thm}\label{Toeplitz}
The translation algebra $C^*(\N)$ arising from the inclusion of the natural numbers in the integers is isomorphic to the Toeplitz algebra, and moreover the inclusion induces the Toeplitz extension.
\end{thm}

Regarding $\Z$ as an infinite cyclic group, it is equipped with a canonical family of partial translations inherited from the right action of the group on itself. The partial translation algebra of $\Z$ induced by this is, by definition,  the reduced group $C^*$-algebra $C^*_\rho(\Z)$. This is generated by a single element,  $\sigma_1$, the \emph{bilateral} shift on $\ell^2(\Z)$ induced by the partial translation (actually a translation) $n\mapsto n+1$.

The subspace $\N$ (which for our purposes will include 0) inherits a family of partial translations by restricting, and corestricting the translations of $\Z$ to $\N$. That is, the set of partial translations on $\N$ consists of all maps $t_n$ defined on $\N$ of the form $t_n(j)=j+n$ where $n\geq 0$, along with all maps of the form
$$t_n^{-1}\colon \{n,n+1,n+2,\dots\} \to \N, t\colon j\mapsto j-n,$$ where $n>0$. 

The corresponding partial translation algebra is by definition the $C^*$-algebra $C^*(\N)\subset B(\ell^2(\N))$ generated by the partial isometries $\tau_i$ and $\tau_{-i}$ corresponding to the partial translations   $s\colon \N \to \N\setminus\{0, \ldots, i-1\}$, $s(j)=j+i$ and  $s^{-1}\colon \N\setminus \{0, \ldots, i-1\} \to \N$, $s^{-1}(j)=j-1$.

\begin{lem} 
The partial translation algebra $C^*(\N)$ is generated by $\tau_1$ (and its adjoint) where $\tau_1$ acts on $\ell^2(\N)$ as a \emph{unilateral} shift. It contains the algebra of compact operators on $\ell^2(\N)$.
\end{lem}

\begin{proof}
It is clear that for each $n$ the operator $\tau_1^n$ is induced by the partial translation $t_n$, while $(\tau_1^*)^n$ is induced by $t_n^{-1}$ proving the first part of the lemma. The operator $\tau_1^*\tau_1-\tau_1\tau_1^*$ is the projection onto the subspace spanned by $0$, and conjugating by the operators $\tau_i^n$ we obtain all the matrix elements, so $C^*(\N)$ contains the algebra of compact operators. 
\end{proof}


Since one can cancel pairs $\tau_1^*\tau_1$, a general element of $C^*(\N)$ of the form 
$$\tau_{1}^{i_1}(\tau_{1}^*)^{j_1}\tau_{1}^{i_2}(\tau_{1}^*)^{j_2}\dots \tau_{1}^{i_k}(\tau_{1}^*)^{j_k}$$
 can be reduced to $\tau_{1}^i(\tau_{1}^*)^j$.


Suppose that $n=i-j$ is positive. Then it is easy to see that $t_it_j^{-1}$ is a subtranslation of $t_n$, that is, its domain is a subset of the domain of $t_n$ and where both are defined they are equal. Moreover the domains differ only by a finite set. Hence as operators $\tau_1^i(\tau_1^*)^j$ and $\tau_1^n$ differ by a finite rank operator. Similarly if $i-j=n$ is negative then $t_it_j^{-1}$ is a subtranslation of $t_n$, and the operators  $\tau_1^i(\tau_1^*)^j$ and $(\tau_1^*)^{n}$ again differ by a finite rank operator.

Thus elements of the form $\tau_1^n$ and $(\tau_1^*)^n$ along with finite rank operators span a dense subspace of $C^*(\N)$. It is easy to see that $\tau_1^i$ and $(\tau_1^*)^j$ never differ by a compact operator (there is no cancellation between them) while $\tau_1^i$ and $\tau_1^j$ differ by a compact operator only if $i=j$. Hence the map $\tau_1^n\mapsto \sigma_1^n$, $(\tau_1^*)^l\mapsto (\sigma_{1}^*)^n$, and $k\mapsto 0$ for every  compact operator $k$, extends to a well defined linear map from a dense subspace of $C^*(\N)$ to $C^*_\rho(\Z)$. This map is moreover a *-algebra homomorphism, and extends by continuity to a *-homomorphism from $C^*(\N)$ to $C^*_\rho(\Z)$ with kernel consisting of compact operators, giving us an extension
$$0\to \K(\ell^2(\N))\to C^*(\N) \to C^*_\rho(\Z) \to 0.$$


We now make the following identifications. The Hilbert space $\ell^2(\N)$ is naturally identified with the Hardy space $H_2$, by taking $e_n$ to $z^n$ for $n\in \N$. Similarly $\ell^2(\Z)$ is naturally identified with $L^2(S^1)$, again the map takes $e_n$ to $z^n$ (however this is now for all $n\in\Z$). With these identifications $\tau_1$ is identified with $T_z$, the Toeplitz operator associated with the identity map   $z: S^1\to S^1$, while the generator $\sigma_1$ of $C^*_\rho(\Z)$ is identified with $M_z$, the operator of pointwise multiplication by the function $z$. Since $C^*(\N)$ and $\T$ are generated by $\tau_1$ and $T_z$ respectively, the above identification of Hilbert spaces gives an isomorphism $C^*(\N)\cong \T$. Similarly we have $C^*_\rho(\Z)\cong C(S^1)$.

The isomorphisms $C^*(\N)\cong \T$ and $C^*_\rho\Z\cong C(S^1)$ identify this with the Toeplitz extension,

$$0 \to \K(H_2) \to \T \to C(S^1) \to 0.$$

\subsection{The translation algebra $C^*(\Z\setminus\{0\})$}

From now on we will abuse notation, denoting both a partial translation and the operator that it defines with the same symbol, using context to determine the meaning. Hence if $s$ is a partial translation we may write $s^*$ to denote the adjoint of the operator corresponding to $s$.

We next consider the effect of removing a single point from the group $\Z$. 

\begin{thm}\label{Trivial} The partial translation algebra $C^*(\Z\setminus\{0\})$ is  a trivial extension of $C(S^1)$ by the compact operators which is therefore not equivalent (in the sense of $K$-homology) to the Toeplitz extension. 
\end{thm}

\begin{proof}
Let $X=\Z\setminus\{0\}$. The partial translations on $X$ that we obtain by restricting and corestricting the translations of $\Z$ have the form
$$s_n\colon \Z\setminus \{-n,0\} \to \Z\setminus \{0,n\},\quad s_n \colon j\mapsto j+n.$$
Note that $s_0$ is the identity and $s_{-n}=s_n^*$ for all $n\in\Z$.

It appears \emph{a priori} that we need all of these partial translations to generate the algebra $C^*(\Z\setminus\{0\})$, since it is not true in this case that $s_n=(s_1)^n$. In fact we will see that it suffices to have $s_0=1,s_1$ and $s_2$.

Consider $s_1s_1^*$. This acts as the identity on $i$ for all $i\neq 1$, while it is undefined at 1. Thus, as an element of the algebra, $s_1s_1^*=1-p_1$ where $p_1$ denotes the rank 1 projection onto the basis element $e_1$ in $\ell^2(X)$. Hence the algebra $C^*(X)$ contains the rank 1 projection $p_1$. Now for any $i,j\in X$, the matrix element $e_{i,j}$ is given by $s_{i-1} p_1 (s_{j-1}^*)$, hence the algebra contains all matrix elements, and hence all compact operators.

Now consider compositions of the form $s_{i_1}s_{i_2}\dots s_{i_k}$. Where this is defined it translates by $l=i_1+\dots+i_k$. In other words it is a subtranslation of $s_l$. The domain on which this is defined is
$$\Z\setminus\{0,-i_k,-(i_{k-1}+i_k), \dots, -l\}$$
in particular it differs from the domain of $s_l$ by only finitely many points. As before $s_i-s_j$ is compact only if $i=j$ and hence we deduce that we have an extension of the form
$$0\to \K(\ell^2(X))\to C^*(X)\to C^*_\rho(\Z)\to 0,$$
where the map $C^*(X)\to C^*_\rho(\Z)$ is given by extending linearly the map taking $s_l$ to $[l]$ and vanishing on the compacts.

We can identify the algebra $C^*(X)$ more explicitly as follows. Consider the partial translation $s_1$. This takes $\Z\setminus\{-1,0\}$ to $\Z\setminus\{0,1\}$. We can extend it to a globally defined translation $t$ on $X$ by defining $t(j)=s_1(j)$ for $j\in X, j\neq -1$ and $t(-1)=1$. Note that this is a compact perturbation of $s_1$ and hence lies in the algebra. Moreover, $t^n$ is a compact perturbation of $s_n$ for all $n$, thus $C^*(X)$ is generated by $t$ along with all compact operators.

Now consider the algebra generated by $t$ alone. Let $\phi\colon X \to \Z$ be the bijective coarse equivalence defined by $\phi(j)=j$ for $j>0$ and $\phi(j)=j+1$ for $j<0$. If $U$ denotes the corresponding unitary from $\ell^2(X)$ to $\ell^2(\Z)$ then $UtU^*$ is right translation by 1, i.e.\ it is the element $[1]$ in $C^*_\rho(\Z)$. Hence using $U$ to identify these two Hilbert spaces, the algebra $C^*(t)$ is identified with $C^*_\rho(\Z)$, the algebra of compacts $\K(\ell^2(X))$ is identified with $\K(\ell^2(\Z))$ and hence (since these together generate $C^*(X)$ we deduce that $C^*(X)$ is identified with the sum $C^*_\rho(\Z)+\K(\ell^2(\Z))$. This identifies the extension as
$$0\to \K(\ell^2(\Z))\to C^*_\rho(\Z)+\K(\ell^2(\Z))\to C^*_\rho(\Z)\to 0.$$

As a $K$-homology cycle for $C^*_\rho(\Z)$ this extension is $(\ell^2(\Z),\rho,1)$ where $\rho$ denotes the right regular representation. This cycle is degenerate, hence it is a trivial element in $K$-homology. However the K-homology class represented by the Toeplitz extension is non-trivial, thus the two extensions are not equivalent.
\end{proof}

\subsection{Coarsely disconnected subspaces}

We will now consider subspaces of $\Z$ such as $\{i^2 : i \in \N\}$, having arbitrarily large gaps. These are said to be coarsely disconnected. 

\begin{prop} Let $X$ be a subset of $\Z$ which is coarsely equivalent to $\{i^2 : i \in \N\}$. Then $C^*(X)$ is isomorphic to the unitised compacts $\widetilde{\K}(l^2(X))$.
\end{prop}

\begin{proof}
Subspaces of $\Z$ which are coarsely equivalent to $\{i^2 : i \in \N\}$ can be characterised as follows. Let $X_+$ denote the non-negative part of $X$ and let $X_-$ denote the negative part of $X$. Then $X_+$ is either finite or consists of an increasing sequence $a_i$ of points with $a_{i+1}-a_i$ tending to infinity as $i\to\infty$. Similarly $X_-$ is either finite or consists of an decreasing sequence $b_i$ of points with $b_i-b_{i+1}$ tending to infinity as $i\to\infty$, and $X_+,X_-$ cannot both be finite.

Now fix some $n\neq 0$, and consider $t_n$ the translation by $n$ on $X$. The domain of $t_n$ consists of those $x\in X$ such that $x+n$ lies in $X$. Since if $X_-$ is infinite, the gaps $b_n-b_{n+1}$ tend to infinity, it follows that if the domain is non-empty then there is a least such $x$. Similarly there is a greatest such $x$, and hence the domain of $t_n$ is finite. Thus as an operator on $l^2(X)$ it follows that $t_n$ is compact. For $n=0$ the translation by $n$ is the identity on $X$, and hence we deduce that $C^*(X)$ is a subalgebra of the unitisation of the compact operators $\widetilde{\K}(l^2(X))$.

We will now show that $C^*(X)$ contains all matrix elements. Pick some $n>0$ for which the domain of $t_n$ is non-empty, let $x$ be the least element of the domain., and let $y$ be the greatest element. Then $y-x=mn$ for some $m>0$, and the translation $(t_n)^m$ takes $x$ to $y$ and is undefined otherwise. Hence as an operator $(t_n)^m$ is the rank 1 operator taking $e_x$ to $e_y$. Now, for any $a,b$ in $X$, the composition $t_{b-y}(t_n)^mt_{x-a}$ is the rank 1 operator taking $e_a$ to $e_b$. The closed span of these operators is $\K(l^2(X))$, hence we conclude that $C^*(X) = \widetilde{\K}(l^2(X))$.

As in the previous examples there is an extension: in this case we see that taking the quotient by the compact operators we obtain a map to $\C$,  since only the identity on $X$ has infinite support, and the quotient here can be regarded as the group $C^*$-algebra of the trivial group. Thus we have the extension
$$0 \to \K(l^2(X)) \to\widetilde{\K}(l^2(X)) \to C^*_\rho(\{0\}) \to 0.$$
\end{proof}

\bigskip

\subsection{The translation algebra of the primes}

The primes inherit a partial translation algebra $C^*(P)$ from the integers and for each positive integer $n$ there is specific element $t_n\in C^*(P)$ represented by translation by $n$. Since there is only one even prime the element $t_1^*t_1$ is a rank 1 projection, and since the translations act transitively the partial translation algebra contains every compact operator. Clearly the element $t_0$ is the identity, so the partial translation algebra of the primes actually contains the unitised algebra
$\widetilde{\K}(l^2(P))$ which is an extension of the form:

$$0 \to \K(l^2(X)) \to\widetilde{\K}(l^2(X)) \to C^*_\rho(\{0\}) \to 0.$$

Now the twin prime conjecture is equivalent to the statement that the operator $t_2$ is not compact. Indeed de Polignac's generalisation of the twin primes conjecture, which asserts the existence of infinitely many prime pairs separated by distance $n$ for each even $n$,  (\cite{dP}) is equivalent to the statement that $t_n$ is not compact for any even $n$. Note that if  de Polignac's conjecture holds for some even $n$ then the algebra $C^*(P)$ is strictly larger than $\widetilde{\K}(l^2(X))$. 

The algebra $\widetilde{\K}(l^2(X))$ has a unique ideal, namely  $ \K(l^2(X))$. If $C^*(P)$  is isomorphic to $\widetilde{\K}(l^2(X))$ it also must contain a unique ideal and this must be $ \K(l^2(X))$ since this is an ideal in $C^*(P)$. It follows that  taking the quotient by this ideal we obtain in both cases a copy of $\C$. Hence $C^*(P)=\widetilde{\K}(l^2(X))$.

\begin{thm}\label{Primes}
The algebra $C^*(P)$ is not isomorphic to $\widetilde{\K}(l^2(X))$ if and only if de Polignac's conjecture holds for some even $n$, if and only if the quotient of $C^*(P)$ by the compact operators is strictly larger than $\C$.
\end{thm}

\section{Translation subspaces of $F_n$}

\subsection{Translation algebras and the Cuntz extension}

In this section we will consider the translation algebras arising from subspaces of the regular 4-valent tree, the Cayley graph of  the free group of rank 2, $F_2$. We will indicate briefly how the arguments carry over to the general case. Let $X$ denote the set of positive words in $F_2$ including the identity. That is $X$ consists of $e$ along with all words in the generators $a$ and $b$.

Consider the partial translations on $X$ defined by $a$ and $b$ (acting by right translation). Then $a$ is globally defined, while the image of $a$ is the set of all positive words ending in $a$. Similarly $b$ is globally defined with image consisting of positive words ending in $b$. The partial translations $a^*$ and $b^*$ are left inverses of $a$ and $b$, respectively. Viewed as an operator on $\ell^2(X)$, the element $a^*$ acts as a bijection from the words ending in $a$ to all words, and acts as zero on $e$ and all words ending with $b$. Similarly for $b$, hence we have the following algebra relations:
$$a^*a=b^*b=1, \quad a^*b=b^*a=0, \quad aa^*+bb^*=1-p_e$$
where $p_e$ denotes the rank 1 projection onto $e$ in $\ell^2(X)$.

A general translation is given by right multiplication by a reduced word $x_1\dots x_k$ with $x_i\in \{a,b,a^{-1},b^{-1}\}$. It is easy to see that as a partial translation $(x_1\dots x_k)$ acts as the composition $y_ky_{n-1}\dots y_1$ where $y_i$ equals $a$ or $b$ respectively if $x_i$ is $a$ or $b$, while $y_i=a^*,b^*$ respectively if $x_i$ is $a^{-1}$ or $b^{-1}$. Note the reversal of the order of composition, due to the fact that we are using right multiplications, e.g.\ the \emph{operator} $ab$ is right multiplication by the \emph{group element} $ba$.

We thus see that the algebra $C^*(X)$ is the $C^*$-algebra generated by $a$ and $b$. (We do not need to explicitly include the identity since $a^*a=1$.)

\bigskip

We now compare $C^*(X)$ with the algebra $C^*(Y)$ of a slightly bigger subset of $F_2$. Let $Y$ denote the set of elements in $F_2$ of the form $a^{n}w(a,b)$, where $n\in\Z$ and $w(a,b)$ is a word in $a$ and $b$. That is, $Y$ consists of reduced words in $a,a^{-1}$ and $b$, where $a^{-1}$ can only occur before the first $b$. Using $Y$ fixes a defect in $X$: each vertex of $X$ has three neighbours with the exception of the identity, which only has two. In $Y$ however \emph{every} vertex has three neighbours, i.e.\ $Y$ is a three regular tree.

We will now consider the algebra $C^*(Y)$. Again the algebra is generated by two elements $a$ and $b$, with $a^*a=b^*b=1$. Let $A$ denote the set of words ending with either $a$ or $a^{-1}$, along with the identity $e$. (Note that the only words ending with $a^{-1}$ are words of the form $a^{-n}$.) Let $B$ denote the set of words ending with $b$. Note that $B$ is the complement of $A$ in $Y$. The partial translation $a$ gives a bijection from $Y$ to $A$ while, $b$ gives a bijection from $Y$ to $B$. The partial translations $a^*$ and $b^*$ are left inverses to $a$ and $b$, hence, as an operator, $aa^*$ is the identity on $A$ and vanishes on $B$, i.e.\ it is the projection of $\ell^2(Y)$ onto $\ell^2(A)$. Conversely $bb^*$  is the projection of $\ell^2(Y)$ onto $\ell^2(B)$.

Thus we conclude that $C^*(Y)$ is a subalgebra of $B(l^2(Y))$ generated by two isometries $a$ and $b$, with the property that
$$aa^*+bb^*=1.$$
This is the defining property of the Cuntz algebra $\O_2$, thus we have $C^*(Y)\cong \O_2$.\footnote{Recall that the Cuntz algebra is constructed concretely as the algebra generated by two such isometries. Cuntz showed that this is the unique algebra with these properties.}

\bigskip
We will now relate $C^*(Y)$ to $C^*(X)$. It will be important to remember at each stage whether a word in $a,b,a^*,b^*$ is to be considered as an operator on $\ell^2(X)$ or on $\ell^2(Y)$. Let $x=x_kx_{k-1}\dots x_1$ be a word in $a,b,a^*,b^*$, considered as an operator on $\ell^2(X)$, and let $y=y_ky_{k-1}\dots y_1$ denote the corresponding operator on $\ell^2(Y)$. We claim that this gives a well defined map from $C^*(X)$ to $C^*(Y)$.

We will say that a word is \emph{quasi-reduced} if it does not contain $a^*a$ or $b^*b$ (since these will act as the identity), and does not contain $a^*b$ or $b^*a$ (since these will act as zero). The quasi-reduced words (as operators on $\ell^2(X)$) span a dense subalgebra of $C^*(X)$.

We claim that the quasi-reduced words are linearly independent. Note that a quasi-reduced word is necessarily of the form $w(a,b)w'(a^*,b^*)$, where $w(a,b)$ (resp.\ $w'(a^*,b^*)$ denotes a positive word in $a,b$ (resp.\ $a^*,b^*$). Say that a word is of type $l$ if $w'$ is a word of length $l$. Note that a quasi-reduced word of type $l$ acts as the zero operator on words in $X$ of length $0,1,\dots l-1$. Thus for a linear combination of quasi-reduced words, the part of type 0 determines the action on $e$. Having removed the type 0 part, the words of type 1 then determine the action on words in $X$ of length 1, etc. Hence considering the action on words of length $0,1,2,\dots$ we deduce that a linear combination which acts as zero, must be zero; that is, the quasi-reduced words are linearly independent.

Now we return to the above map
$$x=x_kx_{k-1}\dots x_1 \in B(\ell^2(X))\mapsto y=y_ky_{k-1}\dots y_1 \in B(\ell^2(Y)).$$
Since the quasi-reduced words are linearly independent, this gives a well defined map from their linear span to $C^*(Y)$. This is a *-algebra homomorphism, and hence contractive, so it extends to a well-defined *-homomorphism from $C^*(X)$ to $C^*(Y)$. Since $a,b$ generate $C^*(Y)$, this homomorphism is surjective.

Clearly the kernel includes $p_e$ since $aa^*+bb^*\mapsto 1$. Hence in fact the kernel includes $\K(\ell^2(X))$, since one can easily construct all matrix elements by pre- and post-composing $p_e$ with translations.

\begin{lem}
Let $x=x_kx_{k-1}\dots x_1$ be a word in $a,b,a^*,b^*$, considered as an operator on $\ell^2(X)$. Consider the corresponding operator $y=y_ky_{k-1}\dots y_1$ on $\ell^2(Y)$, and let $x'$ be the truncation $PyP$ where $P$ is the projection of $\ell^2(Y)$ onto $\ell^2(X)$. Then $x'$ is a compact perturbation of $x$.
\end{lem}

\begin{proof}
Note that as operators on $\ell^2(Y)$, $Py_i$ and $y_iP$ differ only on a single basis vector. Hence $Py-yP$ is a compact operator. Thus $x'=PyP=Py_ky_{k-1}\dots y_1P$ is a compact perturbation of $(Py_kP)(Py_{k-1}P)\dots (Py_1P)$. It now suffices to note that $Py_iP=x_i$, i.e.\ for $a,b$ etc.\ viewed as translations of $Y$, truncating to $X$ gives the corresponding translation on $X$.
\end{proof}

Since we have a right-inverse, which is also a left-inverse modulo compact operators, it follows that the kernel is precisely the compact operators. We thus produce an extension of $\O_2$
$$0\to\K(\ell^2(X))\to C^*(X)\to C^*(Y)\cong \O_2\to 0.$$
This is an analogue for the Cuntz algebra of the Toeplitz extension.

\bigskip

Note that the argument showing that there is a map from $C^*(X)$ to $C^*(Y)$ in fact shows that $C^*(X)$ has a universal property: If $A$ is any $C^*$-algebra generated by two elements $s,t$ satisfying
$$s^*s=t^*t=1\text{ and } s^*t=t^*s=0$$
then there is a surjection from $C^*(X)$ to $A$, taking $a$ to $s$ and $b$ to $t$. These relations suffice to show that there is a homomorphism from the algebra spanned by the quasi-reduced words to $A$, and this extends by continuity to a surjective *-homomorphism.

Another example of an algebra satisfying these relations is the algebra $E_2$ of Cuntz, \cite{Cuntz1}. By definition this is the subalgebra of $\O_3$ generated by the first two isometries $V_1, V_2$, and since $V_1V_1^*+V_2V_2^*+V_3V_3^*=1$, it follows that $V_1^*V_2=V_2^*V_1=0$. Thus there is a surjection from $C^*(X)$ to $E_2$. In  \cite{Cuntz1}, Cuntz showed that there is an extension
$$0\to J_2\to E_2 \to \O_2 \to 0$$
where $J_2$ is an ideal in $E_2$ isomorphic to the algebra of compact operators. Here the quotient map takes the generators of $E_2$ to the generators of $\O_2$. Thus the quotient map $C^*(X)\to C^*(Y)\cong \O_2$ factors through the map $C^*(X)\to E_2$. The kernel of the latter is thus an ideal in the kernel of the former, which is $\K(\ell^2(X))$. Since this is simple, and the kernel is not the whole of $\K(\ell^2(X))$, we deduce that in fact the surjection from $C^*(X)$ to $E_2$ is in fact an isomorphism. Thus we have proved the following theorem.

\begin{thm}\label{Cuntz} There is a canonical isomorphism between the Cuntz extension
$$0\to J_2\to E_2 \to \O_2 \to 0$$
and the extension
$$0\to \K(\ell^2(X))\to C^*(X) \to C^*(Y) \to 0$$
where $X$ and $Y$ are subsets of $F_2$ as above.
\end{thm}

\subsection{The algebras $\mathcal O_n$}

Following the construction in the previous section we replace the free group $F_2$ with $F_n$. Again we define the subtree $X$ to be spanned by all positive words, and choosing a generator we extend this to a regular $n+1$-valent tree which we call $Y$. The inclusions of $X$ and $Y$ into the Cayley graph  endow them with partial translation algebras $C^*(X), C^*(Y)$ respectively and we obtain isomorphisms $C^*(X)\cong E_n$ and  $C^*(Y)\cong \mathcal O_n$, with the algebras defined by Cuntz in \cite{Cuntz1}. By the same argument as above the inclusion of $X$ into $Y$ can be shown to induce the extension

$$0\to J_n\to E_n \to \O_n \to 0.$$

\subsection{Embedding $C_\rho^*F_2$ in $\O_2$}

In \cite{Choi} it was shown that there is an injection of the reduced $C^*$-algebra $C_\rho^*F_2$ into $\O_2$. We conclude by constructing such an inclusion explicitly using our identification of $\O_2$ with $C^*(Y)$. To do so we will construct explicit injective quasi-isometries from $F_2$ into $Y$. It is well known that for any $n,m>2$ the $n$-regular tree is quasi-isometric to the $m$-regular tree, and such quasi-isometries are easy to construct. The purpose of the construction given here is that these quasi-isometries are defined in such a way that they behave well with respect to the natural partial translations on $F_2$ and on $Y$. 

An element $g\in F_2$ is uniquely determined by a geodesic segment emanating from the identity element, and this is encoded by a sequence of turns in the Cayley graph of $F_2$.  To formalise what we mean by a turn we consider the extension of $F_2$ by the cyclic group of order $4$ generated by a rotation of the Cayley graph around the base point as follows.

Let $\alpha$ be the automorphism of $F_2$ taking $a$ to $b$ and $b$ to $a^{-1}$. We will use the notation $x^\alpha$ for the image of $x$ under $\alpha$. Let $H$ be the group of automorphisms of $F_2$ generated by $\alpha$, and let $G$ be the semi-direct product $H\ltimes F_2$. This group is generated by $a,b,\alpha$ with the relations $\alpha^4=1$, $\alpha a \alpha^{-1}=b$ and $\alpha b \alpha^{-1}=a^{-1}$. In general we have $\alpha x\alpha^{-1}=x^\alpha$, for $x$ a word in $F_2$.

An element of $F_2$, viewed as a subgroup of $G$, can be written uniquely in the form $a^{-n}\alpha^{i_0}a\alpha^{i_1}a\dots \alpha^{i_{d-1}}a\alpha^{i_d}$, where $n,d\geq 0$, $i_j\in\{-1,0,1\}$ for $j<d$, $i_0+\dots+i_d=0$ mod $4$, and if $n>0$ then $i_0\neq 0$. A reduced word in $F_2$ can be directly transcribed in this form, and we note that the condition that the word is reduced translates directly into the restrictions that $i_0\neq 0$ if $n>0$, and that there are no factors of $\alpha^2$, except possibly for the final term $\alpha^{i_d}$.

We will now consider a couple of examples. The word $aba$ is equal to the normal form word $a^{0}\alpha^{0}a\alpha^{1}a\alpha^{-1}a\alpha^{0}$. Geometrically we can interpret this as saying that starting from an initial heading of East (the $a$ direction) we go forwards ($\alpha^{0}$) then turn left and move forward one unit ($\alpha^{1}$) then turn  right and move forward one unit ($\alpha^{-1}$). Our final heading is East  and this may be read from the terminal $\alpha^0$. The word $a^{-1}b^2$ is equal to $a^{-1}\alpha^{1}a\alpha^{0}a\alpha^{-1}$, which geometrically we interpret as moving backwards by 1 while facing East, turning left and moving one unit ($\alpha^{1}$) then continuing  forwards for one unit ($\alpha^{0}$). Note that at the end we are facing North, which may also be read from the terminal $\alpha^{-1}$. Backwards moves are special in the following sense: they can only occur as initial moves, and they do not change the direction in which we are facing. We will call the direction in which we are facing at any stage the \emph{heading}.

We define a heading function $h\colon F_2 \to \Z/4\Z$ by $h(a^{-n}\alpha^{i_0}a\dots a\alpha^{i_d})=-i_d$. As $i_0+\dots+i_d=0$ mod $4$, we have $h(x)=i_0+\dots+i_{d-1}$. This justifies the observation above that our final heading can be read from the terminal exponent of $\alpha$. Indeed we note that for a word $x$ in $F_2$, if $x=a^{-n}$ for $n\geq 0$ then $h(x)=0$, otherwise $h(x)$ determines the final term in the word: if $x=x_1\dots x_k$ then $x_k=\alpha^{h(x)}a\alpha^{-h(x)}$.

To embed $F_2$ into $Y$, we will need to encode the turns $\alpha^{i_j}$ and the headings $h(x)$ as words in $a,b$. Define $u_0=a,u_1=b^2,u_{-1}=ab$, and define $v_0=a^2,v_1=b^2,v_2=ab$ and $v_3=ba$ (the index is interpreted modulo 4). We can now define an embedding of $F_2$ into $Y$
$$\phi_0\colon F_2\to Y,\quad \phi_0\colon x=a^{-n}\alpha^{i_0}a\alpha^{i_1}a\dots \alpha^{i_d}\mapsto a^{-n}u_{i_0}u_{i_1}\dots u_{i_{d-1}}v_{h(x)}.$$
This expression for $\phi_0(x)$ is not necessarily a reduced word, however at most there is cancellation of one factor of $a$ from $u_{i_0}$ with one factor of $a^{-1}$. One can read off $i_d,i_{d-1},\dots i_0$ from the right of the word, hence one can recover the original word $x$. Thus $\phi_0$ is injective.




We can extend $\phi_0$ to a map from $G$ to $Y$ in the following simple way. A general element of $G$ is of the form $a^{-n}\alpha^{i_0}a\alpha^{i_1}a\dots \alpha^{i_d}$, where $n,d\geq 0$, $i_j\in\{-1,0,1\}$ for $j<d$, and if $n>0$ then $i_0\neq 0$. Note that we now drop the requirement that $i_0+\dots+i_d=0$ mod $4$. We will call this the normal form for an element of $G$. We define 
$$\Phi\colon x=a^{-n}\alpha^{i_0}a\alpha^{i_1}a\dots \alpha^{i_d}\mapsto a^{-n}u_{i_0}u_{i_1}\dots u_{i_{d-1}}v_{h(x)}$$
where as before $h(x)=-i_d$. Again we can recover the word $x$ from its image, thus $\Phi$ is injective. Moreover it is a bijection: for any element $y$ of $Y$ we can read off a corresponding word $x$ such that $\Phi(x)=y$.

The restriction of $\Phi$ to $F_2$ in $G$ is $\phi_0$. Moreover $G$ decomposes as four left cosets of $F_2$, and these are preserved by the right action of $F_2$ on $G$. Using the bijection $\Phi$ we can define a corresponding action of $F_2$ on $Y$. This action is free and has four orbits which are identified with the cosets by $\Phi$. 
Given the set of orbit representatives $I=\{v_0,v_1,v_2,v_3\}$ the space $l^2(Y)$ is thus identified as $l^2(F_2)\otimes l^2(I)$, and the action of $F_2$ on the right gives rise to the the representation $\rho\otimes 1$ where $\rho$ is the right regular representation of $l^2(F_2)$. This the natural embedding of $C^*_\rho(F_2)$ into $C^*_\rho(G)$. Furthermore the bijection $\Phi$ induces an isomorphism $\Phi_*$ between the bounded linear operators on $\ell^2(G)$ and those on $\ell^2(Y)$. 

\begin{thm}\label{CuntzEmbedding}
The image of $C^*_\rho(F_2)\otimes 1$ under the map $\Phi_*$ is contained in $C^*(Y)\cong\O_2$.
\end{thm}

\begin{proof}  Note that $C^*_\rho(F_2)\otimes 1$ is generated by the elements $\rho(a)\otimes 1, \rho(b)\otimes 1$. We will construct two elements $t_a, t_b\in C^*(Y)$ and show that these are the images under $\Phi_*$ of $\rho(a)\otimes 1, \rho(b)\otimes 1$ respectively. It will thus follow that the image of $C^*_\rho(F_2)\otimes 1$ is contained in $C^*(Y)$.

Recall that right multiplication by $a, b, a^{-1}, b^{-1}$ in $F_2$  induce partial isometries $a, b, a^*, b^*$ on $\ell^2(Y)$. Let $t_a,t_b$  be the partial isometries defined as follows:
\begin{align*}
t_a=&a^3(a^*)^2+ba(a^*)^2b^*+aba^*b^*a^*b^*+b^2(b^*)^2a^*b^*+a^2ba(b^*)^2+a^2b^2b^*a^*,\\
t_b=&b^2a(b^*)^2+aba^*b^*a^*+a^2a^*(b^*)^2a^*+ba(b^*)^3a^*+b^3aa^*b^*+b^4(a^*)^2.
\end{align*}

Viewed as partial translations our operators are chosen so that for any element $y\in Y\subset F_2$, there is a unique term in $t_a$ which is defined at $y$, and likewise for $t_b$. 
Moreover we will see that $t_a$ and $t_b$ are bijections from $Y$ to itself. In terms of the algebra $C^*(Y)$ this means that $t_a$ and $t_b$ are unitaries. 

The following tables, show which term in $t_a,t_b$ acts on any given word in $y$, and how the word is changed.

\[
\left.\begin{array}{|c|c|c|}\hline \hbox{Word ends in} & \hbox{Applicable term for $t_a$} & \hbox{Ending replaced by} \\\hline
a^2=v_0 & a^3(a^*)^2& a^3=u_0v_0\\\hline
a^2b=u_0v_2 & ba(a^*)^2b^*& ab=v_2\\\hline
abab=u_{-1}v_2 & aba^*b^*a^*b^* & ba=v_3\\\hline
b^2ab=u_1v_2 & b^2(b^*)^2a^*b^* & b^2=v_1\\\hline
b^2=v_1 & a^2ba(b^*)^2& aba^2=u_{-1}v_0\\\hline
ba=v_3 & a^2b^2b^*a^*&  b^2a^2=u_1v_0\\\hline
\end{array}\right.\]

\[\left.\begin{array}{|c|c|c|}\hline \hbox{Word ends in} & \hbox{Applicable term for $t_b$} & \hbox{Ending replaced by} \\\hline b^2=v_1 & b^2a(b^*)^2 & ab^2=u_0v_1 \\\hline aba=u_0v_3 & aba^*b^*a^* & ba=v_3 \\\hline ab^2a=u_{-1}v_3 & a^2a^*(b^*)^2a^* & a^2=v_0 \\\hline b^3a=u_1v_3 & ba(b^*)^3a^* & ab=v_2 \\\hline ab=v_2 & b^3aa^*b^* & ab^3=u_{-1}v_1 \\\hline a^2=v_0 & b^4(a^*)^2 & b^4=u_1v_1 \\\hline \end{array}\right.\]



For the benefit of the reader we consider the following example. Let $y=\Phi(\alpha^0 a\alpha^1)=u_0v_3$.  Then $t_ay=u_0u_1v_0=\Phi(\alpha^0 a\alpha^1 a\alpha^0)$. Thus the action of $t_a$ on $y$, is the same as the right action of $a$ on $y$, via the identification $\Phi$ of $Y$ with $G$.  Similarly $t_b y= v_3=\Phi(\alpha)$, and we note that the right action of $b$ on $a\alpha$ produces $a\alpha^2a\alpha^{-1}=\alpha^2 a^{-1} a \alpha^{-1}=\alpha$. Thus the action of $t_b$ on $y$ agrees with the right action of $b$ on $y$.

We now consider the general case. Right multiplication by the element $a$ takes a word of the form $a^{-n}\alpha^{i_0}a\alpha^{i_1}a\dots \alpha^{i_d}$ to $a^{-n}\alpha^{i_0}a\alpha^{i_1}a\dots \alpha^{i_d}a\alpha^{0}$. This is in normal form unless $i_d=2$ in which case we have cancellation as $a\alpha^2a=\alpha^2a^{-1}a=\alpha^2$, thus 
$$a^{-n}\alpha^{i_0}a\alpha^{i_1}a\dots \alpha^{i_{d-1}}a\alpha^{i_d}a\alpha^{0} = a^{-n}\alpha^{i_0}a\alpha^{i_1}a\dots \alpha^{i_{d-1}+2}.$$
In terms of the action of $F_2$ on $Y$ we thus find that multiplication by $a$ has the effect of taking a word of the form $yv_0$ to $yu_0v_0$, taking $yv_1$ to $yu_{-1}v_0$, taking $yv_3$ to  $yu_1v_0$ and taking $yu_iv_2$ to $yv_{2-i}$. Thus the translation $t_a$ is precisely the right action of $a$ on $Y$.

Similarly right multiplication by the element $b$ takes a word of the form \[a^{-n}\alpha^{i_0}a\alpha^{i_1}a\dots \alpha^{i_d}\] to $a^{-n}\alpha^{i_0}a\alpha^{i_1}a\dots \alpha^{i_d+1}a\alpha^{-1}$. This is in normal form unless $i_d+1=2$ in which case we have the cancellation
$$a^{-n}\alpha^{i_0}a\alpha^{i_1}a\dots \alpha^{i_{d-1}}a\alpha^{i_d+1}a\alpha^{-1} = a^{-n}\alpha^{i_0}a\alpha^{i_1}a\dots \alpha^{i_{d-1}+1}.$$
Hence in terms of the action of $F_2$ on $Y$ we find that multiplication by $b$ has the effect of taking a word of the form $yv_0$ to $yu_1v_1$, taking $yv_1$ to  $yu_0v_1$, taking $yv_2$ to $yu_{-1}v_1$ and taking $yu_iv_3$ to $yv_{3-i}$. Again one can check that this is precisely the effect of $t_b$. Thus the translation $t_b$ is the right action of $b$ on $Y$.

We conclude that the subalgebra $C^*(t_a,t_b)$ of $C^*(Y)$ is generated by two unitaries which act on $l^2(Y)\cong l^2(F_2)\otimes l^2(I)$ as $\rho(a)\otimes 1$, $\rho(b)\otimes 1$, where $\rho$ is the right regular representation of $F_2$ on $l^2(F_2)$.  This completes the proof.
\end{proof}


We conclude this section by remarking that the above theorem generalises to show that  $C^*_{\rho_G}(G)$ embeds into $C^*(Y)$ where $\rho_G$ denotes the right regular representation of $G$ on $l^2(G)$. Taking $t_a,t_b$ as in the proof of the theorem, we additionally define
$$t_\alpha = ab(a^*)^2+a^2(b^*)^2+b^2a^*b^*+bab^*a^*.$$
Then $t_\alpha$ takes a word of the form $yv_i$ to $yv_{i-1}$. This is precisely the action of $\alpha$ by right multiplication on $Y$, hence the subalgebra $C^*(t_a,t_b,t_\alpha)$ of $C^*(Y)$ is canonically identified as $C^*_{\rho_G}(G)$. We thus have an embedding of $C^*_{\rho_G}(G)$ into $\O_2= C^*(Y)$.

\end{document}